\documentclass[11pt,a4paper]{article}

\usepackage{amsmath,amsthm,amstext,amscd,amssymb,euscript,mathrsfs}
\usepackage{calrsfs}
\usepackage{epsfig}
\usepackage{color}

\newcommand{\R}{\mathbb R}
\newcommand{\N}{\mathbb N}

\newcommand{\E}{\mathbb E}

\newcommand{\si}{\ensuremath{\sigma}}
\newcommand{\om}{\ensuremath{\Omega}}

\newcommand{\pee}{\ensuremath{\mathbb{P}}}

\newcommand{\loc}{\mathcal{L}}

\def\1{{\mathchoice {\rm 1\mskip-4mu l} {\rm 1\mskip-4mu l}
{\rm 1\mskip-4.5mu l} {\rm 1\mskip-5mu l}}}

\newtheorem{theorem}{{\small T}{\scriptsize HEOREM}}[section]
\newtheorem{corollary}{{\bf{\small C}{\scriptsize OROLLARY}}}[section]
\newtheorem{proposition}{{\bf{\small P}{\scriptsize ROPOSITION}}}[section]
\newtheorem{lemma}{{\bf{\small L}{\scriptsize EMMA}}}[section]
\newtheorem{remark}{{\bf{\small R}{\scriptsize EMARK}}}[section]
\newtheorem{definition}{{\bf{\small D}{\scriptsize EFINITION}}}[section]

\renewenvironment{proof}[1]
{\noindent{{\bf{\small{ P}{\scriptsize ROOF}}}.}\hspace{0.1cm} #1} {$\;\qed$\newline}

\newcommand{\beq}{\begin{eqnarray}}
\newcommand{\eeq}{\end{eqnarray}}

\newcommand{\ba}{\begin{align*}}
\newcommand{\ea}{\end{align*}}

\newcommand{\be}{\begin{equation}}
\newcommand{\ee}{\end{equation}}

\newcommand{\bl}{\begin{lemma}}
\newcommand{\el}{\end{lemma}}

\newcommand{\br}{\begin{remark}}
\newcommand{\er}{\end{remark}}

\newcommand{\bt}{\begin{theorem}}
\newcommand{\et}{\end{theorem}}

\newcommand{\bd}{\begin{definition}}
\newcommand{\ed}{\end{definition}}

\newcommand{\bp}{\begin{proposition}}
\newcommand{\ep}{\end{proposition}}

\newcommand{\bc}{\begin{corollary}}
\newcommand{\ec}{\end{corollary}}

\newcommand{\bpr}{\begin{proof}}
\newcommand{\epr}{\end{proof}}

\newcommand{\bi}{\begin{itemize}}
\newcommand{\ei}{\end{itemize}}

\newcommand{\ben}{\begin{enumerate}}
\newcommand{\een}{\end{enumerate}}

\newcommand{\caB}{{\mathcal B}}
\newcommand{\caC}{{\mathscr C}}

\newcommand{\caD}{{\EuScript D}}

\newcommand{\caF}{{\mathcal F}}

\newcommand{\caH}{{\mathcal H}}

\newcommand{\caK}{{\mathcal K}}

\newcommand{\caR}{{\mathcal R}}

\newcommand{\caT}{{\mathcal T}}

\newcommand{\homo}{\widehat{\Omega}}
\newcommand{\dual}{\rightarrow^D}
\newcommand{\hK}{\widehat{K}}

\newcommand{\A}{{\bf A}}
\newcommand{\K}{{\bf K}}
\newcommand{\colb}[1]{\textcolor[rgb]{0,0,0}{#1}}
\newcommand{\col}[1]{\textcolor[rgb]{0,0,0}{#1}}

\newcommand{\colk}{\color{black}}
\newcommand{\colora}[1]{\textcolor[rgb]{0,0,0}{#1}}
\newcommand{\colorrr}[1]{\textcolor[rgb]{0,0,0}{#1}}

\begin{document}
\title{Dualities in population genetics: a fresh look with
new dualities.}
\author{
Gioia Carinci$^{\textup{{\tiny(a)}}}$,
Cristian Giardin{\`a}$^{\textup{{\tiny(a)}}}$,\\
Claudio Giberti$^{\textup{{\tiny(b)}}}$,
Frank Redig$^{\textup{{\tiny(c)}}}$.\\\\
{\small $^{\textup{(a)}}$ Department of Mathematics, University of Modena and Reggio Emilia}\\
{\small via G. Campi 213/b, 41125 Modena, Italy}
\\
{\small $^{\textup{(b)}}$ Department of Sciences and Methods for Engineering,}\\
{\small University of Modena and Reggio Emilia}\\
{\small via Giovanni Amendola 2,  42122 Reggio Emilia, Italy}\\
{\small $^{\textup{(c)}}$ Delft Institute of Applied Mathematics, Technische Universiteit Delft}\\
{\small Mekelweg 4, 2628 CD Delft, The Netherlands}\\
}
\maketitle

\begin{abstract}
We apply our general method of duality, introduced in \cite{gkr}, to models of population dynamics.
The \col{classical} dualities between forward and ancestral processes can be viewed as a change of representation
in the classical creation and annihilation operators, both for diffusions dual to
coalescents of Kingman's type, as well as for models with finite population size.

Next, using $SU(1,1)$ raising and lowering operators, we find new dualities between
\col{the Wright-Fisher diffusion with $d$ types and the Moran model, both in presence and
absence of mutations. These new dualities relates two forward evolutions.
From our general scheme we also identify self-duality of the Moran model.}
\end{abstract}

\newpage

\section{Introduction}
\color{black}
Duality is one of the most important techniques in interacting particle systems
\cite{l}, models of population dynamics \cite{greven, mhole, huillet}, mutually
catalytic branching \cite{myt} and general Markov process theory \cite{ethierkurtz}.
See \cite{jk1} for a recent review paper containing an extensive list of references, going
even back to the work of L\'{e}vy.

 In all the   interacting particle systems models (e.g. symmetric exclusion, voter model, contact process, etc.) about which we know fine details, such as complete ergodic theorems or explicit formulas for time-dependent correlation functions,
a non-trivial duality or self-duality relation plays a crucial role. This fact also becomes more and more apparent
in recent exact formulas for the transition probabilities of the asymmetric exclusion process
\cite{tracywidom}, \cite{borodin}, where often the starting point is a duality of the type first
revealed in this context by Sch\"{u}tz in \cite{Schutz2}.
In \cite{spohn} the notion of ``stochastic integrability'' is coined, and
related to duality.
It is therefore important to gain
deeper understanding of ``what is behind dualities'', i.e., why some processes admit nice dual processes
and others not, and where the duality functions come from.
This effort is not so much a quest of creating a general abstract framework of duality. It is rather a quest to
create a workable constructive approach towards duality, and using this creating both new dualities in known contexts as well as new Markov processes with nice duality properties.
\color{black}

In the works \cite{gkr,gkrv} duality between two stochastic
processes, in the context of interacting particle systems
and non-equilibrium statistical mechanics, has been
related to a change of representation of an underlying Lie algebra.
More precisely, if the generator of a Markov process is built
from lowering and raising operators (in physics language
creation and annihilation operators) associated to a Lie algebra,
then different representations of these operators give rise to
processes related to each other by duality.
The intertwiner between  the different representations is exactly
the duality function.
Furthermore, self-dualities \cite{grv} can be found using symmetries
related to the underlying Lie algebra (see also \cite{schutz}, \cite{sl}).

The fact that generators of Markov processes can be built from raising and lowering operators
is a quite natural assumption. In interacting particle systems,
the dynamics consists of removing particles at certain places and putting
them at other places. If the rates of these transitions are appropriately
chosen, then the operators of which the effect is to remove or to add a particle
(with appropriate coefficients), \color{black} together with their commutators, generate a Lie algebra \color{black}. For diffusion processes
the generator is built from a combination of multiplication operators
and (partial) derivatives. For specific choices, these correspond
to differential operator representations of a Lie algebra. If this
Lie algebra  also possesses a discrete representation, then this
can lead to a duality between a diffusion process and a process of
jump type, such as the well-known duality between the Wright-Fisher
diffusion and the Kingman's coalescent.

It is the aim of this paper to show that this scheme of finding dual processes
via a change of representation can be applied in the context of mathematical
population genetics. First, we give a fresh look at the classical dualities
between processes of Wright-Fisher type and their dual coalescents.
These dualities correspond to a change of representation in the
creation and annihilation operators generating the Heisenberg algebra.
For population models in the diffusion limit (infinite population size limit)
the duality comes from the standard representation of the Heisenberg
algebra in terms of the multiplicative and derivative operators $(x, d/dx$),
and another discrete representation, known as the Doi-Peliti
representation. The intertwiner
is in this case simply the function $D(x,n)=x^n$.
In the case of finite population size, dualities
arise from
going from a finite-dimensional representation (finite
dimensional creation and annihilation operators
satisfying the canonical commutation relations) to
the Doi-Peliti representation. The intertwiner is exactly
the hypergeometric polynomial found e.g. in \cite{huillet}, \cite{hillion2011natural}, and
gives duality between the Moran model with finite population size
and the Kingman's coalescent.

Next we use the $SU(1,1)$ algebra to \col{find} previously unrevealed dualities between the discrete Moran
model  and the Wright-Fisher diffusion, as well
as self-duality of the discrete Moran model.

These are in fact applications of the previously found dualities between the Brownian energy process and
the symmetric inclusion process, as well as the self-duality of the symmetric inclusion process,
which we have studied in another context in \cite{gkr}, \cite{gkrv}, \cite{grv}.
Put into the context of population dynamics, these dualities give new results for the
multi-type Moran model, as well as the multi-type
Wright-Fisher model.

The rest of our paper is organized as follows. In section
\ref{abstract-setting} we give the general setting and view on duality.
Though many elements of this formalism are already present in previous
papers, we find it useful to put these together here in a unifying, more transparent
and widely applicable framework.
In section \ref{Heisenbergsection}
we discuss dualities in the context of the Heisenberg algebra. This leads
to dualities between diffusions and discrete processes, as well
as between different diffusion processes, and finally between different discrete processes.
In section \ref{Classicalsection} we show how this gives the dualities
between forward (in time) population processes and their ancestral dual coalescents.
In section \ref{sulsection} we apply the $SU(1,1)$ algebra techniques
in the context of population dynamics, finding
new dualities, this time between two forward (in time) processes:
duality between the Wright-Fisher diffusion and the Moran model, and duality of the Moran model with itself.
We give three concrete computations using these new dualities as an illustration.
\section{Abstract Setting}
\label{abstract-setting}

\subsection{Functions and operators.}
Let $\om, \homo$ be metric spaces.
 We denote $\caF(\om), \caF(\homo), \caF(\om\times\homo)$ a space of real-valued functions
from $\Omega$ (resp. $\homo, \om\times\homo$).
Typical examples to have in mind are
$\caC(\om)$, $\caC_c(\om)$ ,$\caC_0(\om)$
the sets of continuous real-valued functions on $\om$ , resp.
continuous real-valued functions with compact support on $\om$, and
continuous real-valued functions on $\om$ going to zero at infinity.
\color{black}
In what follows, this choice of function space is not so stringent, and
can be replaced if necessary by other function spaces such as $L^p$ spaces
or Sobolev spaces.

\noindent
For a function
$\psi:\om\times\homo\to\R$
and linear operators $K: \caD(K)\subset\caF(\om)\to\caF(\om)$, $\hK: \caD(\hK)\subset\caF(\homo)\to\caF(\homo)$ we define
the left action of $K$ on $\psi$ and the right action of $\hK$ on $\psi$ via
\be\label{leftright}
(K_l \psi)(x,y)= \left(K\psi(\cdot, y)\right)(x)\;, \qquad (\hK_r\psi)(x,y)= (\hK\psi(x,\cdot)) (y)\;,
\ee
where we assume that $\psi$ is such that
these expressions are well-defined, i.e. $\psi(\cdot,y)\in \caD(K), \psi(x,\cdot)\in\caD(\hK)$.
\noindent
An important special case to keep in mind is $\om = \{1,\ldots,n\}$ and $\homo=\{1,\ldots,m\}$ finite
sets, in which case we identify functions \colk{on $\om$ or $\homo$ with column vectors and functions on $\om\times \homo$ with $n\times m$ matrices}.
In that case, an operator $K$ (resp.\ $\hK$) on functions on $\om$ (resp.\ $\homo$) coincides with a $n\times n$
(resp. $m\times m$) matrix. Denoting by $\mathbb{K}$ (resp. \ $\widehat{\mathbb{K}}$) such matrices, one has
\be\label{left}
K_l \psi (x, y)= \sum_{z\in\om} \mathbb{K}(x,z) \psi(z, y)= (\mathbb{K}\psi)(x, y)\;,
\ee
\be\label{right}
\hK_r\psi(x,y) =\sum_{u\in\homo} \widehat{\mathbb{K}}(y,u) \psi(x,u)= (\psi \widehat{\mathbb{K}}^T)(x,y)\;.
\ee
Namely, left action of $K$ corresponds to \col{left} matrix multiplication
and right action of $\hK$ corresponds to \col{right} multiplication with the transposed matrix.
The same picture
arises when $\om,\homo$ are countable sets.

\noindent
For two operators $K_1,K_2$ working on the same domain we denote, as usual, $K_1 K_2$ their product or composition, i.e.
\be
(K_1 K_2 f) (x)= \left(K_1 (K_2f)\right) (x)\;
\nonumber
\ee
and
\be
[K_1,K_2]=K_1K_2 - K_2K_1
\nonumber
\ee
the commutator of $K_1$ and $K_2$.
\color{black}
In order to be well-defined, in particular  in the case of {unbounded operators},
we assume that $K_1, K_2$ are working on a common domain
of functions ${\bf D}$, such that $K_i {\bf D}\subset {\bf D}$ for $i=1,2$.

\color{black}

More precisely, we abbreviate $\caB(\om)$
for algebras of linear operators
\colk{working on a common domain $\bf D$,
i.e. ${\bf D} \subset {\caD}(K)$ for all $K\in \caB(\om)$.
 This common domain $\bf D$ is left invariant by the operators, i.e.
$K(\bf D)\subset \bf D$ for all $K\in \caB(\om)$.
\color{black}
}
\color{black} So in such a context expressions
like $\prod_{i=1}^nK_i^{n_i}$ are well-defined
($n_i\in \mathbb{N}$).
The most important contexts in which we naturally have such a common domain and an algebra of operators
working on it are the following.
\begin{enumerate}
\item Formal differential operators working on smooth functions $f:\R^d\to\R$ with compact support
(or smooth functions from a bounded subset of $\R^d$).
Both differential operators, and multiplication with a polynomial
keep this domain invariant. Moreover, in most cases, if an operator belonging to this class
generates a semigroup of contractions (defined via the Hille-Yosida theorem), then
this common domain is a core, i.e., the graph closure of the operator working
on ${\bf D}$ coincides with the generator. See e.g. \cite{ethierkurtz},\cite{nagel}
for details on generators and semigroups, and corresponding diffusion processes. See also \cite{l} for standard arguments
to extend duality from generators to contraction semigroups.
\color{black} Subtleties can arise for domains with boundary: in that case is important to specify boundary conditions
to fix the closure of the operators. These issues will not be dealt with in general here, but on a case-to-case basis
later when we work with diffusions on domains.\color{black}
\item Context of finite set: in this case all functions belong to ${\bf D}$, and
operators can be identified with matrices, i.e., in that case
$\caB(\Omega)$ is just a subalgebra of matrices.
\item Context of a countable set $\Omega=\N_0$, where $\N_0$ denotes the set of positive integers, with zero included. In that case, the operators
we have in mind are finite difference operators, and multiplication operators, working
on e.g. a common domain of functions $f:\N_0\to\R$ going to zero faster than any polynomial
as $n\to\infty$.
\end{enumerate}
\color{black}

These algebras are often representations of an abstract Lie algebra $\caH$
(such as e.g.\ the Heisenberg algebra, see  Sec. 3 and {\cite{Hall}}).
For a general algebra $\caH$, we define the dual
algebra $\caH^*$ as the algebra with the same elements as in $\caH$ but
with product ``$*$'' defined by $a*b= b\cdot a$, where $\cdot$ is the product in $\caH$.
A typical example in the finite dimensional setting is the algebra
of $n\times n$ matrices, where the map $A\to A^T$ maps the algebra into
the dual algebra ($(AB)^T= B^TA^T$).

\bd
For a function
$\psi: \om\times\homo\to\R$ we say that $\psi$ is left exhaustive if
\col{the relation $K_l \psi=0$ implies $K=0$}, and correspondingly we call $\psi$ right
exhaustive if \col{the relation $\widehat{K}_r\psi=0$ implies $\widehat{K}=0$}.
\ed
Notice that in the context of finite sets $\Omega, \widehat{\Omega}$ being right or left exhaustive  just means that the matrix associated to $\psi$ is invertible. In particular we must then require $|\Omega|=|\widehat{\Omega}|$.

\subsection{Duality}
\colk{ We begin this section  with the standard definition of duality \cite{ethierkurtz,l}.
\bd
\label{standard}
Suppose  $\{X_t\}_{t\ge0}$,  $\{\widehat{X}_t\}_{t\ge0}$ are  Markov processes with state spaces  $\om$ and $\widehat{\om}$ and $D: \om\times\homo\to\R$ a bounded
measurable function. The processes  $\{X_t\}_{t\ge0}$,  $\{\widehat{X}_t\}_{t\ge0}$ are said to be dual with respect to $D$  if
\be\label{standarddualityrelation1}
\E_x D(X_t, \widehat{x})=\widehat{\mathbb{E}}_{\widehat{x}} D(x, \widehat{X}_t)\;,
\ee
for all $x\in\om, \widehat{x}\in \hat{\om}$ and $t>0$. In (\ref{standarddualityrelation1})
$\E_x $ is the expectation with respect to the
 law of the $\{X_t\}_{t\ge0}$ process started at $x$, while $\widehat{\mathbb{E}}_{\widehat{x}} $
denotes expectation with respect to the law of the $\{\widehat{X}_t\}_{t\ge0}$ process
initialized at $\widehat{x}$.
\ed
}
\color{black}
Since in most of the examples of duality between
two processes, this property is equivalent with duality of the corresponding
generators, we focus here on dualities between linear operators.
\footnote{An exception is the duality between Brownian motion with reflection and
Brownian motion with absorption, because in that case the duality function $D(x,y)= I(x\leq y)$ (with $I$ denoting the
indicator function) is not in the domain
of the generator. See \cite{jk1} for details on dualities of this type.}

Indeed, in most
examples, duality of the generators in turn follows from
duality between more elementary ``building blocks'' (such as derivatives
and multiplication operators). So in the next section, we focus on
duality between general operators and show how from that notion, which is
conserved under sums and products, it
is natural to consider duality between algebras of operators.
\color{black}

\bd
\label{dualdef2}
Let $K\in \caB(\om)$,\ $\hK\in\caB(\homo)$ and
$D: \om\times\homo\to\R$.
Then we say that $K$ and $\hK$
are dual to each other with duality function $D$
if
\be
\label{dualk}
K_l D= \hK_r D\;
\ee
\colk{where we assume that both sides are well defined, i.e. $D( \cdot, \widehat{x})\in \caD(K)$ for all $\widehat{x}\in \widehat{\om} $ and $D(x,\cdot)\in \caD(\hK)$ for all $x\in \om$.}
We denote this property by
$K\dual \hK$.
\ed
\noindent
In the following we collect elementary but important properties
of the relation $\dual$.
\bt\label{generalprop}
Let $K_1,K_2\in \col{\caB(\Omega)}$, $\hK_1, \hK_2\in \col{\caB(\homo)}$. Suppose that
$K_1\dual \hK_1$, $K_2\dual \hK_2$, and further $c_1,c_2\in \R$
then we have
\begin{enumerate}
\item $\hK_1\rightarrow^{\col{\tilde{D}}} K_1$, with $\col{\tilde{D}(\widehat{x},x)}=D(x,\widehat{x})$.
\item $c_1K_1+c_2K_2\dual c_1\hK_1+ c_2\hK_2$.
\item $K_1K_2\dual \hK_2\hK_1$, in particular $K_1^n\dual \hK_1^n$, $n\in\N$.
\item
$[K_1,K_2]\dual [\hK_2,\hK_1]=-[\hK_1,\hK_2]$, \colk{i.e., commutators of
dual operators are dual}.


\item If $S\in \col{\caB(\Omega)}$ commutes with $K_1$, then
$K_1\rightarrow^{S_l D} \hK_1$, and if $\widehat{S}\in \col{\caB(\homo)}$ commutes
with $\hK_1$, then
$K_1\rightarrow^{\widehat{S}_r D}\hK_1$.
\item If for a collection $\{K_i, i\in I\}\subset\caB(\om)$, and
$\{ \hK_i:i\in I\}\subset\caB(\homo)$ we have
$K_i\dual \hK_i$
then every element of the algebra generated
by $\{K_i, i\in I\}$
is dual to an element of the algebra generated
by $\{ \hK_i:i\in I\}$. More precisely:
\be\label{algrel}
K_{i_1}^{n_1}\ldots K_{i_k}^{n_k}\dual \hK_{i_k}^{n_k}\ldots \hK_{i_1}^{n_1}
\ee
for all $n_1,\ldots, n_k\in \N$,
and for constants $\{c_i : i\in I\}$
\[
\sum_{i\in I} c_i K_i \dual \sum_{i\in I} c_i\widehat{ K}_i\;.
\]
\item Suppose \col{(only in this item)} that
$K_1\in \caB(\Omega_1)$,
$\hK_1\in\caB(\homo_1)$,
$K_2\in \caB(\Omega_2)$,
$\hK_2\in\caB(\homo_2)$.
If $K_1\rightarrow^{D_1} \hK_1$ and $K_2\rightarrow^{D_2} \hK_2$, then
$K_1\otimes K_2 \rightarrow^{D_1\otimes D_2} \hK_1\otimes \hK_2$, where
\[
D_1\otimes D_2 (x_1,x_2;\widehat{x}_1,\widehat{x}_2)= D_1(x_1,\widehat{x}_1)D_2 (x_2,\widehat{x}_2).
\]
\item If $K$ and $\hK$ generate Markov semigroups $S_t=e^{tK}$ and $\widehat{S}_t= e^{t\hK}$,
{and if $D: \om\times\homo\to\R$ are functions such that
$D( \cdot, \widehat{x}), \widehat{S}_tD(\cdot,\widehat{x}) \in \caD(K)$ for all $\widehat{x}\in\widehat{\Omega}$
and $D(x, \cdot), {S}_tD(x, \cdot) \in \caD(\widehat{K})$ for all ${x}\in{\Omega}$,}  then
$K\dual\hK$ implies $S_t\dual \widehat{S}_t$. If moreover these
semigroups correspond to Markov processes $\{X_t,t\geq 0\}, \{\widehat{X}_t,t\geq 0\}$ on $\om, \widehat{\om}$, the
relation $S_t\dual \widehat{S}_t$ reads in terms of these processes:
\be\label{standarddualityrelation}
\E_x D(X_t, \widehat{x})=\widehat{\mathbb{E}}_{\widehat{x}} D(x, \widehat{X}_t)\;,
\ee
for all $x\in\om, \widehat{x}\in \hat{\om}$ and $t>0$.
\end{enumerate}
\et
\bpr
The properties listed in the theorem are elementary and
their proof is left to the reader. The only technical issue is item 8, i.e., passing
from duality on the level of generator to duality of the corresponding semigroups and
processes. This result is obtained by using the uniqueness of the semigroup.
More precisely, if $K$ (or the closure of $K$) generates
a semigroup $S_t$ (formally denoted by $e^{tK}$) then for all
functions $D(\cdot,\widehat{x})$ which are in the domain of $K$ for all
$\widehat{x}\in\widehat{\Omega}$,
the unique solution of the equation
\be
\label{bat}
\frac{d}{dt} f_t(x,\widehat{x}) = K_l f_t(x,\widehat{x})
\ee
with initial condition $f_0(x,\widehat{x})= D(x,\widehat{x})$ is given
by $f_t(x,\widehat{x})= (S_t)_l D(x,\widehat{x})$.

On the other hand, the relation $K\dual\hK$ implies that also
$f_t(x,\widehat{x}) = (\widehat{S}_t)_r D(x,\widehat{x})$ solves
the same equation. Indeed,  since $\widehat{S}_t$ has generator
$\widehat{K}$, it follows that
\[
\frac{d}{dt} (\widehat{S}_t)_r D(x,\widehat{x})
=  (\widehat{S}_t)_r \widehat{K}_r D(x,\widehat{x}) = (\widehat{S}_t)_r {K}_l D(x,\widehat{x}) =
{K}_l (\widehat{S}_t)_r  D(x,\widehat{x})
\]
where in the last equality we used that $\widehat{S}_tD(\cdot,\widehat{x}) \in \caD(K)$.

As the equation \eqref{bat} has a unique solution if follows that
$S_t\dual \widehat{S}_t$.
For more details see \cite{l} theorem 4.13, Chapter 3, pag. 161
in the context of spin systems
and also \cite{nagel, jk1} for more general cases.
\epr
\br
Item 6 of theorem \ref{generalprop} is useful
in particular if the collection
$\{K_i, i\in I\}\subset\caB(\om)$ is a generating set for the algebra.
Then every element of the algebra has a dual by \eqref{algrel}, and it suffices
to know dual operators for the generating set to
infer dual operator for a general element of the algebra.
In practice, one starts from such a generating set and the commutation relations
between its elements (defining the algebra) and associates to it by a single duality
function a set of dual operators with the same commutation relations up to a change of sign
(cfr. item 4).
In other words, one moves via the duality function from a representation
of the algebra to a representation of the dual algebra.
\er
\br
The relation \eqref{standarddualityrelation} is the form in which one usually formulates
duality between two Markov processes (cfr. Def. \ref{standard}).
Remark however that the relation
$S_t\dual \widehat{S}_t$ between the semigroups is more general. It may happen that
$S_t$ is a Markov semigroup, whereas $\widehat{S}_t$ is not. E.g., mass can
get lost in the evolution according to the dual semigroup $\widehat{S}_t$, which
means $\widehat{S}_t 1\not= 1$, or it can happen that
$\widehat{S}_t$ is not a positive operator \colk{(see e.g. Remark 4.2 of \cite{CGGR} where the duality between the generator of the symmetric exclusion process and a non-positive differential operator is exhibited, and e.g. \cite{l} chapter III, section 4,  for duality with Feynman Kac factors in the context of spin systems)}.

\er

In \cite{mhole}, the author studies the so-called duality space associated to
two operators. In our notation, this is the set
\be\label{dualityspace}
\caD(K,\widehat{K})= \{ D:\Omega\times\widehat{\Omega}\to\R \;| \; K\rightarrow^D\widehat{K}\}\;.
\ee
Our point of view here is {\em instead to consider for a fixed duality function the set
of pairs of operators $(K,\widehat{K})$ such that $K\dual \widehat{K}$.}
These operators then usually form a representation and a dual representation of
a given algebra, with $D$ as intertwiner {(see, however, Eq. \eqref{symdualrel})}.

In theorem \ref{generalprop} item 5, we see that we can produce new duality functions
via ``symmetries'', i.e., operators commuting with $K$ or $\widehat{K}$.
In the context of finite sets $\Omega= \widehat{\Omega}=\{1,\ldots,n\}$, we have more:
if there exists an invertible duality function $D$, then all other duality functions
are obtained via symmetries acting on $D$.
We formulate this more precisely in the following proposition.
We will give examples in the \col{subsequent} sections.
\bp\label{symprop}

\noindent
\ben
\item Let $\Omega= \widehat{\Omega}=\{1,\ldots,n\}$, and let $K\rightarrow^D\widehat{K}$. Suppose
furthermore that the associated $n\times n$ matrix  $\mathbb{D}$ is invertible.
Then, if $K\rightarrow^{D'}\widehat{K}$, we have that there exists $S$ commuting
with $K$ such that $D'=SD$.
\item For general $\Omega, \widehat{\Omega}$ we have the following.
Suppose $D$ and $D'$ are duality functions for the duality between
$K$ and $\widehat{K}$. Suppose furthermore that $D'=S_l D$, for
some operator $S$. Then we have
\be
(KS-SK)_l D=0\;.
\ee
In particular if $D$ is left exhaustive, then we conclude $[S,K]=0$.
Similarly, if $D'=\widehat{S}_r D$ then
\be
(\widehat{K}\widehat{S}-\widehat{S}\widehat{K})_r D=0\;.
\ee
and if $D$ is right exhaustive, then we conclude
$[\widehat{S},\widehat{K}]=0$\;.
\een
\ep
\bpr
In the proof of the first item, with slight abuse of notation, we use the notation $K,\widehat{K}, D, S$ both
for the operators and for their associated matrices.
We have, by assumption
\[
KD= D\widehat{K}^T.
\]
By invertibility of $D$, $S=D'D^{-1}$
is well-defined
and we have
\[
SK= D' D^{-1}K = D' \widehat{K}^T D^{-1}\;,
\]
and
\[
KS= KD' D^{-1}=  D' \widehat{K}^T D^{-1}\;,
\]
hence $[K,S]=0$.

\noindent
For the second item, use $K\dual \widehat{K}, K\rightarrow^{D^\prime} \widehat{K}$ to
conclude
\[
(KS)_l D= K_l(S_l D)= K_l D'= \widehat{K}_r D' \;,
\]
as well as
\[
(SK)_l D= S_l \widehat{K}_r (D)= \widehat{K}_r (S_l D)= \widehat{K}_r D' \;,
\]
Hence
\[
([S,K])_l D=0.
\]
\col{The remaining part of the proof (for the right action case) is similar}.
\epr
\subsection{Self duality}

Self-duality is duality between an operator and itself, i.e., referring
to definition \ref{dualdef2}:
$\Omega=\widehat{\Omega}$ and $K=\widehat{K}$. The corresponding
duality function such that $K\dual K$ is then a function $D:\Omega\times\Omega\to\R$
and we call it a self-duality function.
For self-duality, of course, all the properties listed in theorem \ref{generalprop} hold.

\noindent
In the finite case $\Omega=\{1,\ldots,n\}$, a self-duality function  is a $n\times n$
matrix and self-duality reads, in matrix form,
\[
{K}D= D{K}^T\;.
\]
Therefore, in this setting such a matrix $D$ can always be found because every matrix
is similar to its transposed
\cite{trans}, i.e., self-duality always holds with an invertible $D$.

Other self-duality functions can then be found by acting
on a given self-duality function with symmetries of $K$
(i.e.
operators $S$ commuting with $K$), as we derived in
item 5 of theorem \ref{generalprop} and in
proposition \ref{symprop}.
In particular we have in this finite context, in the notation \eqref{dualityspace}:
\be\label{symdualrel}
\caD(K,K)= \{ SD: [S,K]=0\}\;,
\ee
with $D$ an arbitrary invertible self-duality function. So this means that the correspondence
between self-duality functions and symmetries of $K$ is one-to-one. This characterization of
the duality space has the advantage that the set of operators commuting with a given operator
is easier to identify.

\noindent
In the finite setting, if $K$ is the generator (resp. transition operator) of
a continuous-time (resp. discrete-time) Markov chain, then if this Markov chain
has a reversible probability measure $\mu:\om \to [0,1]$, a duality function for self-duality is given
by the diagonal matrix
\[
D(x,y) = \delta_{x,y} \frac{1}{\mu(x)}\;.
\]
This is easily verified from the detailed balance relation $\mu(x) \mathbb{K}(x,y)=\mathbb{K}(y,x)\mu(y)$.
This ``cheap'' duality function is usually not very useful \col{since it is diagonal},
 but it can be turned in a more ``useful'' one by acting
with symmetries. All the known self-duality functions in discrete interacting
particle systems such as the exclusion process, independent random walkers,
the inclusion process, etc. can be obtained by this procedure \cite{gkr}.

\section{Dualities in the context of the Heisenberg algebra.}\label{Heisenbergsection}
\colk{The (abstract) Heisenberg algebra $\caH(m)$ \cite{Hall} is an algebra generated by
$2m$ elements $\K_i$, $\K_i^{\dagger},\, i=1,\ldots, m,$ satisfying the following  commutation relations:
\be
\label{heisenberg-commutation}
[{\K}_i,\K_j]=0,\quad [\K_i^{\dagger},\K_j^{\dagger}]=0,\quad [\K_i,\K_j^{\dagger}]=\delta_{i,j} {I},\quad i,j=1,\ldots, m
\ee
where $I$ is the unit element of $\caH(m)$. Relations (\ref{heisenberg-commutation}) are called {\em canonical commutation relations}.}

In this section we focus on representations of
Heisenberg algebra and its dual algebra, and the corresponding duality functions that
connect these different representations.

\subsection{Standard creation and annihilation operators}
As a first example \colk{of a representation of $\caH(1)$}, let us start with the operators $A^\dagger,A$
working on smooth functions $f:\R\to\R$ with compact support,
defined as
\be
\label{heisenberg-differential}
A f(x)= f'(x),\qquad A^{\dagger} f(x)= xf(x) \;,
\ee
\colk{in physical jargon: the annihilation  and  creation  operators.}
These operators  satisfy the canonical commutation relations \colk{(\ref{heisenberg-commutation}) for $m=1$ with $\K_1=A, \K^\dagger_1=A^\dagger$.
Indeed $[A, A^\dagger]=I$, where $I$ is the identity operator, while the remaining relations are trivially satisfied.}

\noindent
The same commutation relations \colk{(\ref{heisenberg-commutation})}, up to a negative sign,
can be achieved using operators working on discrete
functions. Considering
\be
\label{heisenberg-discrete}
a f(n)= n f(n-1),\qquad  a^{\dagger} f(n)= f(n+1)\;,
\ee
acting on functions \colb{$f:\N_0\to\R$}, we have
$[a, a^\dagger]=-I$. Therefore, in view of the item
4 of theorem \ref{generalprop} and Remark 2.1, \col{the operators
$a,a^\dagger$ are natural candidates
for duality with $A,A^\dagger$.}

\noindent
To find $D$ such that
$A\dual a$,
we use the definition \ref{dualdef2}:
$$
A_l D(x,n)= D'(x,n) = a_r D(x,\col{n})= n D(x,n-1)\;,
$$
which yields
$$
D(x,n) = \sum_{k=0}^n {n \choose k} c_{n-k} \; x^k
$$
with $\{c_i : i \in \N \}$ a sequence of constants.
In the same way, the duality condition $A^{\dagger}\dual a^{\dagger}$,
produces
$$
A^{\dagger}_l D(x,n)= x D(x,n) = a^{\dagger}_r D(x,\col{n})= D(x,n+1)\;,
$$
which gives
$$
D(x,n) = x^n D(x,0)\;,
$$
with $D(x,0)$ an arbitrary function.
Therefore, if we want both dualities to hold with the same duality function,
then we are restricted to the choice
$$
D(x,n) = c_0 x^n\;.
$$
Without loss of generality we \col{can} choose $c_0 =1$.

As a consequence,  by using item  6 of theorem \ref{generalprop},
we obtain the following result.
\begin{theorem}\label{polthm}
For $0\le n \le m$, let $\alpha_n:\R\to\R$ be a finite sequence of polynomials.
The differential operator $K$ defined on smooth functions with compact support, {$f\in \mathcal C_0^\infty(\R)$}, of the form
\[
K=\sum_{n=0}^m \alpha_n(x) \frac{d^n}{dx^n} = \col{\sum_{n=0}^m \alpha_n(A^{\dagger}) A^{n}}
\]
is dual with duality function $D(x,n)=x^n$ to \colk{the operator $\hat K$ acting on the space of real valued functions  $f: \N_0 \to \R$, $f=\{f_n\}_{n\in \N_0}$ :}
\[
\hK= \sum_{n=0}^m a^n \alpha_n (a^\dagger)\;
\]
where the operators $A, A^{\dagger}$ are defined in (\ref{heisenberg-differential})
and  $a, a^{\dagger}$ are defined in (\ref{heisenberg-discrete}).
\end{theorem}
\br
Using the Doi-Peliti method, in \cite{ok}, some results of the type of theorem \ref{polthm} are obtained for
reaction diffusion systems.
\er
\color{black}
\colk{We close this section with a representation of the Heisenberg algebra $\caH(m)$ with $m>1$, which generalizes the previous one to functions of several variables and that will be used in the next sections. The generators of this representation, working on the smooth functions $f:\R^m\to\R$ with compact support, are
\be
\label{heisenberg-differential-m}
A_i f(x)= \frac{\partial}{\partial x_i} f(x),\qquad A_i^{\dagger} f(x)= x_if(x) \;, \quad i=1,\ldots,m,
\ee
 also in this case they are called annihilation and creation operators.  Clearly $A_i$ and $A_i^{\dagger}$ satisfy the canonical commutation relations  (\ref{heisenberg-commutation}), thus they generate a representation of the Heisenberg algebra $\caH(m)$.\\
In analogy with  (\ref{heisenberg-discrete})  we introduce a discrete algebra generated by the operators acting on
functions \colb{$f:\N_0^m\to\R$} via
\be\label{heisenberg-discrete-m}
a_i f(n) = n_if(n-e_i),\quad  a^\dagger_i f(n)=f(n+e_i), \quad i=1,\ldots,m
\ee where
\colb{$n\in \N_0^m$ and $e_i\in \N_0^m$} is the $i$-th canonical unit vector defined via $(e_i)_j=\delta_{i,j}$. On the basis of the previous  discussion we have that, for each $i$ the dualities $A_i\rightarrow^{D_i} a_i, A^\dagger_i\rightarrow^{D_i} a^\dagger_i$  hold with duality function $D_i(x_i,n_i)=x_i^{n_i}$. Thus by
Theorem \ref{generalprop}, item 7, we have duality between the tensor products of the generators of the continuous representation  ($\otimes_{i=1}^m K_ i$ with $K_i\in \{A_i, A_i^\dagger\}$) and the tensor products of the generators of the discrete one ($\otimes_{i=1}^m \hK_ i$ with $\hK_i\in \{a_i, a_i^\dagger\}$). The duality function is given by
\be\label{heisenberg-duality-x-n}
D(n,x) =\prod_{i=1}^m D_i(x_i,n_i)= \prod_{i=1}^m x_i^{n_i}\;
\ee
that is $\otimes_{i=1}^m K_ i\rightarrow^D \otimes_{i=1}^m \hK_i$.}

\subsection{Generalization}
In the following proposition we show how to generate the duality functions
for more general generators ${\bf A}, {\bf A^\dagger}$ of a representation
of the Heisenberg algebra, namely by repetitive action of the creation
operator on the ``vacuum'' which is annihilated by the operator $\col{\bf A}$.
\bp\label{aprop}
Suppose $ [{\bf A},{\bf A^\dagger}]=I$, and let $D(x,n)$ be functions such that
\beq\label{assumptionsss}
&&\col{({\bf A}_{\col{l}}^\dagger)^n} D(x,0)= D(x,n)\;,
\nonumber\\
&&
{\bf A}_{\col{l}} D(x,0)= 0\;,
\eeq
then $\A\dual a$ and $\A^\dagger \dual a^\dagger$, where $a,a^\dagger$ are
the discrete representation defined in \eqref{heisenberg-discrete}.
As a consequence, for a finite sequence \col{of polynomials} $\alpha_n$, \col{with}
$0\leq n\leq m$, we have the analogue of \col{theorem} \ref{polthm}:
\[
\sum_{n=0}^m \alpha_n (\A^\dagger) \A^n \dual \sum_{n=0}^m a^n \alpha_n(a^\dagger)\;.
\]
\ep
\bpr
\col{We have $\A^\dagger\dual a^\dagger$ by the assumption on $\A^{\dagger}$ in \eqref{assumptionsss}
and the definition of $ a^\dagger$ in \eqref{heisenberg-discrete}.}
We therefore have to prove $\A\dual a$.
Start from the commutation relation $[\A,\A^\dagger]=I$ to write
\[
\A (\A^\dagger)^n= (\A^\dagger)^n \A +[\A, (\A^\dagger)^n]= (\A^\dagger)^n \A+n (\A^\dagger)^{n-1}.
\]
Then use the assumptions \eqref{assumptionsss} to deduce
\begin{eqnarray*}
\A_{\col{l}} D(x,n)&=& \A_{\col{l}} (a_{\col{r}}^\dagger)^n D(x,0)= \A_{\col{l}} (\A_{\col{l}}^\dagger)^n D(x,0)
\\
&=& (\A_{\col{l}}^\dagger)^n \A_{\col{l}} D(x,0) +n (\A_{\col{l}}^\dagger)^{n-1} D(x,0)= n (\A_{\col{l}}^\dagger)^{n-1} D(x,0)
\\
&=&
n (a_{\col{r}}^\dagger)^{n-1} D(x,0)
\\
&=&
n D(x,n-1)
= a_{\col{r}} D(x,n)\;.
\end{eqnarray*}
\epr

As an application, we can choose linear combinations of multiplication
and derivative
\be
{\bf A}= c_1x + c_2 \frac{d}{dx}, \qquad {\bf A}^{\dagger}= c_3x+ c_4 \frac{d}{dx},
\ee
with the real constants satisfying $c_2c_3-c_1c_4=1$, then we satisfy the
commutation relation $[{\bf A},{\bf A}^\dagger]=I$.
To find the corresponding duality function that ``switches'' from ${\bf A},{\bf A}^\dagger$ to $a, a^\dagger$, we
start with
\[
{\bf A}_l D(x,0)= c_1x D(x,0)+ c_2 D'(x,0)= a_r D(x,0)=0
\]
which gives as a choice
\[
D(x,0)= \exp\left(-\frac{c_1}{c_2}\frac{x^2}{2}\right)
\]
and next,
\[
D(x,n)=({\bf A}^{\dagger}_l)^n D(x,0)= \left(c_3x + c_4 \frac{d}{dx}\right)^n D(x,0).
\]
An important particular case (related to the harmonic oscillator in quantum mechanics and the
Ornstein-Uhlenbeck process), is when  $c_1=c_2=1/2$ and $c_3=-c_4=1$. With this choice one finds that
the duality function is
$D(x,n) = e^{-x^2/2}H_n(x)$, where $H_n$ is the Hermite polynomial of order $n$.


\subsection{Dualities with two continuous variables}
Within the scheme described in section \ref{abstract-setting}
we can also find dualities between two operators
both working on continuous variables, as the following
example shows.

\noindent
Consider again the operators $A,A^{\dagger}$ in (\ref{heisenberg-differential}).
A ``dual'' commutation relation (in the sense of item 4 of theorem \ref{generalprop})
can be obtained by considering a copy of those operators and exchanging their
role. Namely, we look for dualities $d/dx\dual y, x\dual d/dy$. Imposing that the
left action of $d/dx$ (resp. $x$) does coincide with the right action of $y$ (resp. $d/dy$)
one finds the duality function $D(x,y) =e^{xy}$. As a consequence one immediately has
the following
\begin{theorem}
\col{For $0\le n \le m$, let $\alpha_n:\R\to\R$ be a finite sequence of polynomials.}
A differential operator working on smooth functions of the
real variable $x$ and with the generic form
\[
K=\sum_{n=0}^m \alpha_n (x) \frac{d^n}{dx^n}
\]
is dual, with duality function $D(x,y)=e^{xy}$, to the
operator working on smooth functions of the real variable $y$
given by
\[
\hK= \sum_{n=0}^m y^n \alpha_n (\frac{d}{dy})\;.
\]
\end{theorem}
\noindent
\color{black}
As a first simple illustration, consider the operator $\frac{1}{2}\frac{d}{dx^2}$, which
is dual to the multiplication operator $\frac{y^2}{2}$ (with duality function $e^{xy}$). The semigroup
with generator $\frac{1}{2}\frac{d}{dx^2}$ is Brownian motion, and the ``semigroup''
generated by $\frac{y^2}{2}$ is of course multiplication with $e^{ty^2/2}$. As a consequence, {denoting by $(W_t)_{t\ge 0}$  the standard Brownian motion}, we have
for the corresponding semigroups:
\[
\E_x \left[e^{yX_t}\right]=\E \left[e^{y(x+W_t)}\right]= \widehat{\E}_{y} \left[e^{Y_t x}\right]= e^{\frac{t y^2}{2}} e^{xy}
\]
which in this case can of course be directly verified from the equality $\E [e^{yW_t}]=e^{\frac{t y^2}{2}} $.
\color{black}

If we specify that the operators work on functions $f:[0, \infty)\to\R$, we can use the
duality $d/dx\dual -y, x\dual -d/dy$ with duality function $D(x,y)= e^{-xy}$ and
when the operators can be interpreted as pregenerators of diffusions one has
the following
\begin{corollary}
The diffusion pregenerator
\[
\loc = (c_1 x^2+ c_2 x) \frac{d^2}{dx^2} + (c_3 x) \frac{d}{dx}
\]
with $c_1>0, c_2\geq 0$
on the domain
\[
\caD(\loc)= \{ f: [0,\infty)\to\R: f, f', f''\in \caC([0,\infty)), \loc f(0)=0\}
\]
is dual to
\[
\widehat{\loc} = c_1 y^2 \frac{d^2}{dy^2} + \left(-c_2 y^2 + c_3 y\right) \frac{d}{dy}
\]
on the same domain, with duality function $D(x,y)= e^{-xy}$. For the corresponding diffusion processes
$\{X_t: t\ge 0\}$, $\{Y_t: t\ge 0\}$ we thus have
\be\label{boeam}
\E_x e^{-yX_t}= \widehat{\E}_y e^{-xY_t}\;.
\ee
The particular case $c_2=0$ gives  that the diffusion pregenerator $c_1x^2 d^2/dx^2+ c_3 x d/dx$ is self-dual.
\end{corollary}
\br
Notice that naively applying $d/dx\dual y, x\dual d/dy$ with duality function  $D(x,y)=e^{xy}$
in the previous context yields that
\[
\loc = (c_1 x^2+ c_2 x) \frac{d^2}{dx^2} + (c_3 x) \frac{d}{dx}
\]
is dual
to
\[
\widehat{\loc} = c_1 y^2 \frac{d^2}{dy^2} + \left(c_2 y^2 + c_3 y\right) \frac{d}{dy}
\]
with duality function $D(x,y)= e^{xy}$
but it might be that the corresponding relation on the level of the semigroup
\be\label{boeam2}
\E_x e^{yX_t}= \widehat{\E}_y e^{xY_t}\;.
\ee
does not yield useful information because the corresponding processes do not possess exponential
moments, i.e., the relation reads ``$\infty=\infty$''.
\er
\color{black}

\subsection{Discrete creation and annihilation operators}
The following example starts from a finite dimensional representation of
the Heisenberg algebra \colk{$\caH(1)$} (in the spirit of \col{\cite{Dimakis, hillion2011natural}}).
We consider $\Omega=\Omega_N = \{0,\ldots, N\}$ and
$\widehat{\Omega}=\N$.
For functions $f: \Omega_N\to\R$ we define the operators
\begin{eqnarray}
\label{representation1}
\mathrm{a}_N f(k)
& = &
\left(N-k\right)f(k+1) +\left(2k-N\right)f(k) -k f(k-1)\;,\nonumber \\
\mathrm{a}_N^{\dagger} f(k)
& = &
 \sum_{r=0}^{k-1}  (-1)^{k-1-r} \frac{{N\choose r}}{{N \choose k}} f(r)\;,
\end{eqnarray}
with the convention $f(-1)=f(N+1)=0$.
Consider
\be\label{finiteNdual}
D_N(k,n) = \frac{{k\choose n}}{{N\choose n}}= \frac{k(k-1)\cdots(k-(n-1))}{N(N-1)\cdots(N-(n-1))}
\ee
with the convention $D_N (k,0)=1$, $D_N(k,N+1)=0$.
\col{Let us denote by ${\cal W}_N$ the vector space
generated by the functions
$k\mapsto D_N(k, n), 0\leq n\leq N$.}
\bp
\label{finite}
\beq
(\mathrm{a}_N )_l D_N (k,n)&=& n D_N (k,n-1), \ \forall\ 1\leq n, \forall\ k\geq n-1\;,
\nonumber\\
(\mathrm{a}_N )_l D_N (k,0) &=& 0 \ \forall\ 0\leq k\leq N\;,
\nonumber\\
(\mathrm{a}_N^{\dagger})_l D_N(k,n)&=& D_N(k, n+1) \ \forall\ \ 0\leq n\leq N, k\geq n\;.
\eeq
As a consequence, as operators on \col{${\cal W}_N$} we have
\be\label{disheis}
[\mathrm{a}_N ,\mathrm{a}_N^{\dagger}]=I\;,
\ee
i.e., $\mathrm{a}_N ,\mathrm{a}_N^{\dagger}$ form a finite
dimensional representation of the canonical commutation relations.
\ep
\bpr
Straightforward computation.
\epr

\begin{remark}
Notice that in the limit $N\to\infty$, putting $k/N=x$, and $f(k)= \phi(x)=\phi(k/N)$,
$\mathrm{a}_N f(k)$ converges to $d\phi/dx$.
Next, notice that $D_N (k,n)= \phi^{(n)}_N(x)$, where
\[
\phi^{(n)}_N(x) = x \left(\frac{x-\frac1N}{1-\frac1N}\right)\ldots \left(\frac{x-\frac{n-1}N}{1-\frac{n-1}N}\right)\;,
\]
converges to $x^n$. The effect of the operator $\mathrm{a}_N^{\dagger}$ on
$D_N(k,n)$ is to raise the index $n$ by one.
\col{ Since $D_N(k,n) =\phi^{(n)}_N(x) \to x^n$ for $N\to\infty$ and
$\mathrm{a}_N^{\dagger}D_N(k,n) = A^{\dagger}\phi^{(n)}_N(x) \to x^{n+1}$
for $N\to\infty$, we conclude that
in the limit $N\to \infty$, the operator $\mathrm{a}_N^{\dagger}$
coincides with the multiplication operator $A^{\dagger}$ defined
in \eqref{heisenberg-differential}}.
\end{remark}

\noindent
\col{The discrete finite dimensional representation of the Heisenberg algebra
given in proposition \ref{finite} will be used at the end of section \ref{Classicalsection}
to fit within the scheme of a change of representation the classical duality
between the Moran model and the block-counting process of the Kingman's
coalescent. 
Since we only use the block-counting process of the Kingman's coalescent (rather than the full partition-valued process) 
we will use the name ``Kingman's coalescent'' for that block-counting process here and also later on.}
\color{black}
We end this section with a comment on the relation between the discrete representation in proposition \ref{finite}
and the \col{B}inomial distribution.
This also offers an alternative simple way to see the commutation
relation \eqref{disheis}.
\subsection{Relation with invariant measures}
\colk{In many models  where there is duality or self-duality (see e.g.   Exclusion, Inclusion and Brownian Energy processes), there
\col{exists} a one-parameter family of invariant measures $\nu_\rho$,  (see e.g.  Section 3.1 of \cite{CGGR} for more details)} and integrating
the duality function w.r.t.\ these measures usually gives a simple
expression of the parameter \col{$\rho$}.
In the context of diffusion processes with discrete dual,
this relation is usually
that the duality function with $n$ dual particles
integrated over \col{the distribution} $\nu_\rho$
equals $\rho^n$. A similar relation connects
the polynomials $D_N(k,n)$ \col{in \eqref{finiteNdual}}
to the binomial distribution. This general relation between a natural one-parameter
family of measures and the duality functions cannot be a coincidence
and requires further investigation.

The polynomials $D_N(k,n)$ are (as a function of $k$) indeed naturally associated to the binomial distribution.
Denoting by
\[
\nu_{N,\rho} (k)= {N\choose k} \rho^k (1-\rho)^{N-k}
\]
the binomial distribution with success probability $\rho\in [0,1]$, we have
\be\label{binodual}
\sum_{k=0}^N D_N(k,n)\nu_{N,\rho}(k)= \rho^n\;.
\ee

\noindent
For a function $f:\Omega_N\to\R$ we define its binomial transform
$\caT f:[0,1]\to\R$
by
\be\label{ktrans}
(\caT f)(\rho)= \sum_{k=0}^N f(k) \nu_{N,\rho} (k) \;.
\ee
If for such $f$, we
write its expansion
\[
f(k)= \sum_{r=0}^N \col{c_r}D_N(k,r)
\]
we say that $f$ is of degree $l$ if \col{$c_l\not=0$} and all higher coefficients \col{$c_k, k>l$} are
zero. We then have, \col{using \eqref{binodual} and \eqref{ktrans}},
\[
(\caT f)(\rho) = \sum_{r=0}^N \col{c_r} \rho^r.
\]
The functions $f$ and \col{$\caT f$} have therefore the same components with respect
to two different bases: one given by \col{$\{D_N(k,r), r=0,\ldots, N\}$} which is a base of  $\R^{N+1}$ and the
other given by \col{$\{\rho^r, r=0,\ldots, N\}$} which is a base of the space of polynomials on $[0,1]$
of degree at most equal to $N$.
We then have, for all $f$:
\[
\col{(\caT\mathrm{a}_N f)(\rho)}= \col{(\caT f)^{\prime}(\rho)}
\]
and for all $f$ with degree less than or equal to $N-1$:
\[
\col{(\caT \mathrm{a}_N^{\dagger} f)(\rho)}=\rho \cdot \col{(\caT f)}(\rho)\;.
\]
This relation shows that the operators $\mathrm{a}_N,\mathrm{a}_N^{\dagger}$ after binomial transformation
turn into the standard creation and annihilation operators (\colk{$\rho, d/d\rho$}) for
a restricted set of functions (polynomials of degree at most $N$).

\section{Classical dualities of population dynamics}\label{Classicalsection}

The scheme developed in the section 2, together with the change of representation
discussed in section 3, allows to recover many of the well-know dualities of classical
models of population genetics \col{\cite{as,etheridge,jk}}. We first consider diffusion processes of the
Wright-Fisher diffusion type and then discrete processes
for a finite population of $N$ individuals of the Moran type.
\color{black}
In this section, as well as in section \ref{sulsection} when we consider generators $L$ of diffusion processes on an interval or on
a multidimensional simplex $\Omega$, we will always define them with absorbing boundary conditions, i.e.,
the pregenerator (of which the generator is the graph closure) is defined on the domain
of smooth functions $f$ with compact support such that $Lf$ vanishes on the boundary.
In the case that the boundary is not attainable (such as Wright-Fisher diffusion with mutation, depending on the mutation rate)
the domain of the pregenerator
consists of smooth functions $f$ with compact support contained in the interior of $\Omega$.
See \cite{ethierkurtz} chapter 8, section 1, Theorem 1.4 for more details on generators with absorbing boundary conditions.\\
\color{black}

\noindent
{\bf Diffusions and coalescents.}
\\
Consider smooth functions
$f: [0,1]\to\R$ vanishing at the boundaries $0$ and $1$.
A diffusion process on $[0,1]$ has generator
of the form
\be\label{diff}
\loc = \alpha(x) \frac{d^2}{dx^2} + \beta(x) \frac{d}{dx}
=\alpha(A^\dagger) A^2 +\beta(A^\dagger) A\;,
\ee
\col{with $A$ and $A^\dagger$ defined in (\ref{heisenberg-differential})}. More precisely, we choose
\beq\label{ab}
\alpha(x)&=&
 \sum_{k=1}^\infty \alpha_k x^k\;,
\nonumber\\
\beta(x) &=&
\sum_{k=0}^\infty \beta_k x^k\;,
\eeq
where the coefficients $\alpha_k,\beta_k$ satisfy the following
\beq\label{condcoef}
&&\alpha_2= -\sum_{k\not=2, k=1}^\infty \alpha_k,\qquad \alpha_k\geq 0\quad  \forall k\not= 2 \;,
\nonumber\\
&&\beta_1=-\sum_{k\neq 1, k=0}^\infty \beta_k, \qquad \beta_k\geq 0 \quad \forall k\not=1\;.
\eeq
Typical choices are
$\alpha(x)= x-x^2$, $\beta(x)=(1-x)$.
By the duality $A\dual a$, $A^\dagger\dual a^\dagger$ and by theorem
\ref{polthm} we find that
$\loc$ is dual to
\beq\label{dualdiff}
\hat{\loc} f(n)&=& \left(a^2 \alpha(a^\dagger) + a\beta(a^\dagger)\right) f(n)
\nonumber\\
&=&
n(n-1) \sum_{k =1}^\infty \alpha_k (f(n+k-2)- f(n))
\nonumber\\
&& +  \;n \sum_{k=0}^\infty  \beta_k( f(n+k-1)-f(n))
\eeq
with duality function $D(x,n)=x^n$.
By the conditions (\ref{condcoef}) on the coefficients, this corresponds to a Markov
chain on the natural numbers.

\noindent
We can then list a few examples.
\ben
\item {\bf Wright Fisher neutral diffusion.}
\[
\loc= x(1-x) \frac{d^2}{dx^2}= A^\dagger(1-A^\dagger)A^2\;.
\]
This corresponds to $\beta=0$ and $-\alpha_2=\alpha_1=1$ and gives the dual
\begin{eqnarray*}
\hat{\loc} f(n)&=& \left(a^2 (a^\dagger(1-a^\dagger))\right)f(n)
\\
&=&n(n-1) (f(n-1)-f(n))\;,
\end{eqnarray*}
which is the well-known Kingman's coalescent block-counting process.
\item {\bf Wright Fisher diffusion with mutation.}
\[
\loc = x(1-x) \frac{d^2}{dx^2} + \theta(1-x) \frac{d}{dx}=A^\dagger(1-A^\dagger)A^2+\theta (1-A^\dagger)A\;.
\]
This corresponds to $\alpha_1=-\alpha_2=1$, $\beta_0=-\beta_1=\theta$.
This gives the dual
\begin{eqnarray*}
\hat{\loc} f(n)&=&\left(a^2 (a^\dagger(1-a^\dagger))+ \theta a(1-a^\dagger)\right)f(n)
\\
&=&
n(n-1) (f(n-1)-f(n))+\theta n (f(n-1)-f(n))\;,
\end{eqnarray*}
which corresponds to Kingman's coalescent with extra rate $\theta n$ to go down from
$n$ to $n-1$, due to mutation.
\item {\bf Wright Fisher diffusion with ``negative'' selection.}
\be\label{sel}
\loc= x(1-x) \frac{d^2}{dx^2} - \si x(1-x) \frac{d}{dx}= A^\dagger(1-A^\dagger)(A^2-\si A)
\ee
with $\si>0$,
which corresponds to $\alpha_1=-\alpha_2=1$, $\beta_2=-\beta_1=\si$.
The dual is
\begin{eqnarray}\label{corombo}
\hat{\loc} f(n)&=& \left((a^2-\si a)a^\dagger(1-a^\dagger)\right)f(n)
\\
&=&
n(n-1) (f(n-1)-f(n)) + \si n (f(n+1)-f(n))\nonumber
\end{eqnarray}
\color{black}
Notice that
\be\label{sel2}
\loc= x(1-x) \frac{d^2}{dx^2} + \si x(1-x) \frac{d}{dx}= A^\dagger(1-A^\dagger)(A^2-\si A)
\ee
with $\si>0$ (i.e., ``positive selection'') can also be dealt with. Indeed it is dual to the same process
\eqref{corombo}, but now with duality function $(1-x)^n$, coming from
the representation $(1-x)$, $-d/dx=d/d(1-x)$ of the generators of the Heisenberg algebra. This in turn
corresponds to the transformation $x\mapsto 1-x$.\color{black}
\color{black}
\item {\bf Stepping stone model.}
This is an extension of the Wright-Fisher diffusion, modelling subpopulations of which the individuals have two types, and
which evolve within each subpopulation as in the neutral Wright Fisher diffusion, and additionally, after reproduction a fraction of each subpopulation is exchanged with other subpopulations.
These subpopulations are indexed by a countable set $S$. The variables $x_i\in [0,1], i\in S$ then represent the fraction
of type $1$ in the $i^{th}$ subpopulation.
The generator of this model is defined on smooth
local functions (i.e., depending
 on a finite number of variables) on the set $\Omega=[0,1]^S$
  and given by
\be\label{stepstone}
\loc=\sum_{i,j\in S} p(i,j)(x_j-x_i) \left(\frac{\partial}{\partial x_i}-\frac{\partial}{\partial x_j}\right) + \sum_{i\in S} x_i(1-x_i)
\frac{\partial^2}{\partial x_i^2}\;.
\ee
Here
$p(i,j)=p(j,i)$, with positive entries outside of the diagonal and with $\sum_{j\in S} p(i,j)=1$,
is an irreducible symmetric Markov transition\color{black}\  kernel on the set $S$.

\colk{In terms of the standard creation and annihilation
operators $A_i^\dagger = x_i, A_i = \frac{\partial}{\partial x_i}$, introduced (in the case of $m=|S| $ finite) at the end of Section 3.1,  this generator reads}
\be\label{stepheis}
\loc = \sum_{i,j\in S} p(i,j) (A^{\dagger}_j-A^{\dagger}_i)(A_i-A_j) + \sum_{i\in S} A_i^\dagger(1-A_i^\dagger)A_i^2\;.
\ee

\colk{The dual operators $a_i$ and $a_i^\dagger$ have also been introduced in (\ref{heisenberg-discrete-m}) and the duality function between tensor products of operators  is (\ref{heisenberg-duality-x-n}).}�
As a consequence the generator $\loc$ in \eqref{stepheis} is dual to $\hat{\loc}$ given by
\be\label{dualstepstone}
\hat{\loc} = \sum_{i,j\in S } p(i,j)(a_i-a_j) (a^\dagger_j-a^\dagger_i)+ \sum_{i\in S} a_i^2 a_i^\dagger(1-a_i^\dagger)
\ee
or equivalently the generator  $\loc$ in \eqref{stepstone} is dual to $\hat{\loc}$ given by
\begin{eqnarray}
\label{dualstepstone2}
\hat{\loc} f(n)
& = & \sum_{i,j\in S } p(i,j)n_i(f(n-e_i+e_j)-f(n))\nonumber\\
& + & \sum_{i,j\in S} p(j,i) n_j(f(n+e_i-e_j)-f(n))\nonumber\\
& + & \sum_{i\in S} n_i(n_i-1)(f(n_i-e_i)-f(n))\;,
\end{eqnarray}
which is the generator of a Markov process on $\N_0^S$ with transitions
$n\to n-e_i+ e_j$ (resp. $n\to n-e_j+ e_i$) at rate $n_ip(i,j)$ (resp. $n_jp(j,i)$) and
$n\to n-e_i$ at rate $n_i (n_i-1)$. \colorrr{Here $e_i$ denotes the vector with components
$(e_i)_k = \delta_{i,k}$}. The first type of transitions are of random walk type and correspond
to the exchange of subpopulations, whereas the second type are the transitions corresponding to the Kingmans' coalescent
in each subpopulation.
\een

\noindent
{\bf Finite-size populations \col{\cite{cannings, gladstien1978characteristic}} and coalescents.}\\
As a final example, we illustrate the use of the discrete creation and annihilation
operators $\col{a^\dagger_N}, a_N$, corresponding to population models with
$N$ individuals in the discrete Moran model.
This is the discrete analogue of the neutral Wright-Fisher diffusion
\be
\label{moran2}
\loc_Nf(k) = \frac{N^2}{2} \frac{k}{N}  \left(1-\frac{k}{N}\right) \left(f(k+1) + f(k-1) -2 f(k)\right)\;.
\ee
In terms of the discrete creation and annihilation operators $\mathrm{a}_N, \mathrm{a}_N^{\dagger}$
\col{defined in \eqref{representation1}},  this generator
reads
\be
\col{\loc_N= \mathrm{a}_N^{\dagger}(1-\mathrm{a}_N^{\dagger}) \mathrm{a}_N^2}\;.
\ee
\col{By theorem \ref{generalprop} and proposition \ref{finite},
we therefore obtain immediately that this generator is dual to the generator of the Kingman's
coalescent with duality function \eqref{finiteNdual}}.

\section{$SU(1,1)$ \col{algebra} and corresponding dualities}\label{sulsection}
In this section we show new dualities for models of population dynamics, using
dualities between well-chosen differential operators and discrete
operators. 
These operators have been used in the context of particle systems and models
of heat conduction \cite{gkr}. Interpreted here in terms of population models,
they yield in that context new dualities.

In the whole of this section, the common domain $\bf D$ of the operators that will appear
(as section 2.1) will be 
the set of multivariate polynomials. This set is closed under the action of the operators, and
forms a core of the Markov generators that will appear.
\color{black}

\colk{The results of this section are obtained applying the $SU(1,1)$ algebra \cite{ikg}, which is an (abstract) algebra generated by a set of elements $\{\K_i^+,
\K_i^-,\K_i^o\},\, i=1,\ldots, n$
that satisfy the following commutation relations:
\be\label{su11comA}
[\K_i^o, \K_i^{\pm}] =
\pm \K_i^{\pm} \;,
\qquad\qquad
[\K_i^-,\K_i^+] =
2 \mathcal \K_i^o\;.
\ee
}
We start with
the following two families
(labeled by $m$) of infinite dimensional representations
of the algebra $SU(1,1)$. The first family of operators acts on \colk{smooth}
functions $f: [0,\infty)\to\R$, whereas the second family acts on
functions \colb{$f: \N_0\to\R$}.
\begin{eqnarray}\label{contsu11}
\caK^+ &=& z \;,\nonumber\\
\caK^- &=& z \frac{d^2}{dz^2} +\frac{m}{2} \frac{d}{dz}\;,\nonumber\\
\caK^0 &=& z\frac{d}{dz} +\frac{m}{4}\;,
\end{eqnarray}
and
\begin{eqnarray}\label{discrsu11}
K^+ f(n)&=& \left(\frac{m}{2}+n\right)f(n+1)\;,\nonumber\\
K^- f(n) &=& n f(n-1)\;,
\nonumber\\
K^0 f(n) &=& \left(\frac{m}{4} +n\right) f(n)\;.
\end{eqnarray}
The $\caK$ operators satisfy the $SU(1,1)$ commutation relations \colk{(\ref{su11comA})}
whereas the $K$ operators satisfy the dual commutation relations
(i.e., with opposite sign).
Therefore, the operators are candidates for
a duality relation \colk{(see item 4 of Theorem 2.1 and Remark 2.1)}.

In order to find corresponding duality functions,
we now first give the analogue of Proposition \ref{aprop}
in the context of the $SU(1,1)$ algebra. This tells us that
if $\caK^+$ and $\caK^0$ are dual to their discrete analogues
given in \eqref{discrsu11} with duality function $D$, and
$D(z,0)$ is ``annihilated'' by $\caK^-$ (i.e., $\caK^- D(z,0)=0$), then, using
the commutation relations \eqref{su11comA}, we obtain that
$\caK^-$ and $K^-$ are also dual with the same duality function, and hence, \colk{by item 6 of Theorem 2.1,}
the whole algebra spanned by $\caK^\alpha$ is dual to the algebra spanned by
the $K^\alpha$, $\alpha\in\{+,-,0\}$.
\bp\label{kprop}
Suppose $D(z,n)$ are functions such that
\beq\label{assuK}
\caK_l^+ D(z,n)&=& \left(\frac{m}{2}+n\right)D(z,n+1)
\nonumber\\
\caK_l^0 D(z,n)&=& \left(n+\frac{m}{4}\right) D(z,n)
\nonumber\\
\caK_l^- D(z,0)&=&0
\eeq
where the $\caK^{\alpha}$, for $\alpha\in \{+,-,0\}$, are working on the $z$-variable.
Then we have
$\caK^{\alpha}\dual K^\alpha$ , where
$K^\alpha$ are the discrete operators defined in \eqref{discrsu11}.
\ep
\bpr
By assumption \eqref{assuK} we have
$\caK^{\alpha}\dual K^\alpha$ for $\alpha\in \{+,0\}$.
Therefore we have to prove that $\caK^-\dual K^-$.
In this proof we abuse notation and denote
$\caK^{\alpha} D(z,n) = \caK_l^{\alpha} D(z,n)$.
We start by proving that
\be
\caK^- D(z,1)= D(z,0)
\ee
Using \eqref{assuK} with $n=1$
\beq
\caK^- D(z,1) &=& \caK^- \left(\frac{\caK^+ D(z,0)}{m/2}\right)
\nonumber\\
&=&
\frac2m \caK^-\caK^+ D(z,0)
\nonumber\\
&=&
\frac2m \left(\caK^+\caK^- + [\caK^-,\caK^+]\right)D(z,0)
\nonumber\\
&=&
\frac2m (2\caK^0) D(z,0)
\nonumber\\
&=& \frac2m\left(\frac{m}{2} D(z,0)\right) = D(z,0)
\eeq
Then, we proceed by induction. Assume $\caK^- D(z,n-1)= (n-1) D(z,n-2)$. Then
\beq
\caK^- D(z,n) &=&
\caK^-\left(\frac{\caK^+ D(z,n-1)}{\frac{m}2+n-1}\right)
\nonumber\\
&=&
\frac1{\frac{m}2+n-1}\left(\caK^+\caK^- + [\caK^-,\caK^+]\right) D(z,n-1)
\nonumber\\
&=&
\frac1{\frac{m}2+n-1}\left(\caK^+\caK^- + 2\caK^0\right) D(z,n-1)
\nonumber\\
&=&\frac1{\frac{m}2+n-1}\caK^+\left((n-1)D(z,n-2)\right) + \nonumber \\
& & + \frac1{\frac{m}2+n-1}\left(2n-2 +\frac{m}{2}\right) D(z,n-1)
\nonumber\\
&=&\frac{D(z,n-1)}{\frac{m}2+n-1}\left(\left(\frac{m}{2} +n-2\right)(n-1) + 2n-2 + \frac{m}2\right)
\nonumber\\
&=& n D(z,n-1)
\eeq
Here in the third step we used the commutation relations, in the fourth step the induction hypothesis
and \eqref{assuK},
and in the fifth step \eqref{assuK}.
\epr

To find the duality function $d: [0,\infty)\times\N\to\R$ relating the discrete and continuous
representations \eqref{contsu11} and \eqref{discrsu11} we use the previous proposition: first
\[
{\caK}_{\col{l}}{^-} d(z,0)= \left(z \frac{d^2}{dz^2} +\frac{m}{2} \frac{d}{dz}\right) d(z,0)= K_{\col{r}}^- d(z,0)=0
\]
which gives as a possible choice $d(z,0)= 1$. Then, we can act with $\caK^+$:
\[
(\caK_{\col{l}}^+)^n d(z,0)= z^n = (K_{\col{r}}^+ )^n d(z,0)= \frac{m}{2} \left(\frac{m}{2}+1\right)\ldots \left(\frac{m}{2}+ n-1\right)d(z,n)\;,
\]
and we find
\be\label{dualsu11}
d(z,n)=\frac{z^n}{\frac{m}{2} \left(\frac{m}{2}+1\right)\ldots \left(\frac{m}{2}+ n-1\right)}
= \frac{z^n\Gamma\left(\frac{m}{2}\right)}{\Gamma\left(\frac{m}{2} +n\right)}\;.
\ee
Since $d(z,n)$ is of the form $c_n z^n$ we also see that
\[
\caK_l^0 d(z,n) = \left(z\frac{d}{dz} +\frac{m}{4}\right) d(z,n) = \left(n+ \frac{m}{4}\right) d(z,n)
\]
Then, by proposition \ref{kprop}
$\caK^-_l d(z,n)= K^-_r d(z,n)$.
We can then summarize these findings in the following result.
\bp\label{su11prop}
The family of operators given by \eqref{contsu11} and
the family of operators given by \eqref{discrsu11}
are dual with duality function given by \eqref{dualsu11}.
As a consequence, every element of the algebra generated
by the operators \eqref{contsu11} is dual to
an element of the algebra \col{generated by}
\eqref{discrsu11}, obtained by replacing the operators by their duals
and reverting the order of products.
\ep

\subsection{Markov generators constructed from $SU(1,1)$ raising and lowering operators.}
\label{sub1}
The relevance of the $K^{\pm},\caK^{\pm}$ lies in the fact that some natural
generators of diffusion \col{processes} of population dynamics can be rewritten in terms of them.
As mentioned before (see beginning of section 4), these will be generators of processes on a multidimensional simplex
with {\em absorbing boundary conditions}, i.e., the domain of the pregenerator $L$ consists of
smooth functions on the simplex such that $Lf$ vanishes at the boundary.
\color{black}
We start now with defining these generators.
\begin{definition}[\cite{etheridge}, p. 55]
The $d$-types Wright-Fisher model with symmetric parent-independent mutation
at rate $\theta\in\R$ is a diffusion process on the simplex $\sum_{i=1}^dx_i=1$
defined by the generator
\begin{eqnarray}
\label{dWF-mut}
\mathscr{L}_{d,\theta}^{WF} g(x)
& = &
\sum_{i=1}^{d-1} \frac12 x_i(1-x_i) \frac{\partial^2 g(x)}{\partial x_i^2}
- \sum_{1\le i < j \le d-1} x_ix_j \frac{\partial^2 g(x)}{\partial x_i \partial x_j} \nonumber \\
& + & \frac{\theta}{d-1}\sum_{i=1}^{d-1}(1-dx_i)\frac{\partial g(x)}{\partial x_i}\;.
\end{eqnarray}
\end{definition}
\begin{definition}
The Brownian Energy process with parameter $m\in\R$ on the complete graph
with $d$ vertices \col{(BEP(m))} is a diffusion on $\R_+^d$ with generator
\begin{eqnarray}
\label{bep}
\mathscr{L}_d^{BEP(m)} f(y)
& = &
\frac{1}{2}\sum_{1\le i< j \le d} y_iy_j \left(\frac{\partial}{\partial y_i} - \frac{\partial}{\partial y_j}\right)^2 f(y)\nonumber \\
& - &
\frac{m}{4} \sum_{1\le i < j \le d} (y_i-y_j) \left(\frac{\partial}{\partial y_i} - \frac{\partial}{\partial y_j}\right) f(y) \;.
\end{eqnarray}
\end{definition}
\bp\label{bep=wf}
The Brownian Energy process with parameter $m\in\R$ on the
complete graph with $d$ vertices and  with initial condition $\sum_{i=1}^{d}x_i=1$
coincides with the $d$-types Wright-Fisher model
with symmetric parent-independent mutation at rate $\theta= \frac{m}{4}(d-1)$, i.e.
$$
\mathscr{L}_d^{BEP(m)} f(x_1,\ldots,x_{d-1}, x_d)
=
\mathscr{L}_{d,\frac{m}{4}(d-1)}^{WF}g(x_1,\ldots,x_{d-1})
$$
with
$$
g(x_1,\ldots,x_{d-1}) = f(x_1,\ldots,x_{d-1}, 1-\sum_{j=1}^{d-1}x_j)\;.
$$
\ep
\begin{proof}
The statement of the proposition is a consequence of the fact the BEP evolution conserves
the quantity $x_1+\ldots+ x_d$. Consider the initial condition
$\sum_{i=1}^d x_i =1$ and define the function $\phi: \R^{d-1}\to\R^{d}$
such that
$$
(x_1,\ldots, x_{d-1}) = x \mapsto \phi(x)= (x_1,\ldots,x_{d-1}, 1-\sum_{j=1}^{d-1}x_j)\;.
$$
Then $g(x) = f(\phi(x))$ and, for all $i=1,\ldots,d-1$, using the chain rule gives
$$
\frac{\partial g(x)}{\partial x_i} = \frac{\partial f(\phi(x))}{\partial y_i}  -  \frac{\partial f(\phi(x))}{\partial y_d} \;.
$$
A computation shows that
$$
\mathscr{L}_d^{BEP(m)} f(x_1,\ldots,x_{d-1}, x_d)
=
\mathscr{L}_{d,\frac{m}{4}(d-1)}^{WF}\;g(x_1,\ldots,x_{d-1})\;.
$$
\end{proof}
\begin{definition}[\cite{etheridge-griffiths}, Eq. (12)]
In the $d$-types Moran model with population size $N$ and with  symmetric parent-independent mutation at rate $\theta$,
a pair of individuals of types $i$ and
$j$ are  sampled uniformly at random, one dies with probability $1/2$ and the other reproduces.
In between reproduction events each individual accumulates
mutations at a constant rate $\theta$ and his type mutates to any of the others with
the same probability.
Therefore, denoting types occurrences by $k= (k_1,\ldots,k_{d-1})$, where $k_i$ is the number of individuals of type $i$,
the process has generator
\begin{eqnarray}
\label{dMoran-mut}
& & \mathscr{L}_{N,d,\theta}^{Mor} \;g(k)  = \nonumber\\
& & \frac12 \sum_{1\le i < j \le d-1}
\left[ k_i \left(k_j + \frac{2 \theta}{d-1}\right) \;(g(k-e_i+e_j) - g(k)) \right. \nonumber \\
&&
\qquad\qquad\qquad\left. k_j \left(k_i + \frac{2 \theta}{d-1}\right) \;(g(k+e_i-e_j) - g(k)) \right] \nonumber \\
& & +
\frac12 \sum_{i=1}^{d-1}
\left[\left(N-\sum_{j=1}^{d-1}k_j\right) \left( k_i + \frac{2 \theta}{d-1}   \right) (g(k+e_i) -g(k))\right.\nonumber\\
&&
\qquad\qquad\qquad\left. k_i\left(N-\sum_{j=1}^{d-1}k_j  + \frac{2 \theta}{d-1}   \right) (g(k-e_i) -g(k))\right]\;.\nonumber\\
\end{eqnarray}
\end{definition}
\begin{definition}
The \col{Symmetric Inclusion} process with parameter $m\in\R$ on the complete graph with $d$ vertices \col{(SIP(m))} is a Markov
process on \colk{$\N_0^d$} with  generator
\begin{eqnarray}
\label{sip}
&& \mathscr{L}_{d}^{SIP(m)} f(k) = \nonumber\\
&& \frac12 \sum_{1\le i < j \le d}
\left[k_i \left(k_j + \frac{m}{2}\right) (f(k-e_i+e_j) - f(k)) \right.\nonumber \\
&& \qquad\qquad\quad +\left. k_j \left(k_i + \frac{m}{2}\right) (f(k+e_i-e_j) - f(k))\right]\;.\nonumber\\
\end{eqnarray}
\end{definition}
\bp\label{sip=mor}
The generator of the \col{Symmetric Inclusion} process with parameter $m\in\R$ on the complete graph with
$d$ vertices and with initial condition $\sum_{i=1}^{d}n_i=N$
coincides with the generator of the $d$-types Moran model
with population size $N$ and with symmetric parent-independent mutation at rate $\theta= \frac{m}{4}(d-1)$.
\ep
\bpr
One verifies that
$$
\mathscr{L}_d^{SIP(m)} f(k_1,\ldots,k_{d-1}, k_d)
=
\mathscr{L}_{N,d,\frac{m}{4}(d-1)}^{Mor}\;g(k_1,\ldots,k_{d-1})
$$
with
$$
g(k_1,\ldots,k_{d-1}) = f(k_1,\ldots,k_{d-1}, N-\sum_{j=1}^{d-1}k_j)\;.
$$
\epr
We can now state our duality result.
\begin{theorem}
\label{pippo3}
In the presence of symmetric parent-independent mutation at rate $\theta$,
the $d$-types Wright-Fisher diffusion process with generator
(\ref{dWF-mut})  and the $d$-types Moran model with $N$ individuals
and with generator (\ref{dMoran-mut}) are dual with duality function
\be
\label{dual-fct}
\tilde{D}_N(x,k) = \prod_{i=1}^{d} \frac{x_i^{k_i}}{\Gamma(\frac{2\theta}{d-1} + k_i)}\;,
\ee
with
$$
x_d = 1-\sum_{j=1}^{d-1}x_j\;,
\qquad\qquad
k_d = N-\sum_{j=1}^{d-1}k_j\;.
$$
\end{theorem}
\bpr
The statement of the theorem is a consequence of the duality between
BEP(m) and SIP(m), which we now recall.
We consider the two families of operators representing the $SU(1,1)$ and
dual $SU(1,1)$ commutation relations, now rewritten in $d$ coordinates:
\begin{equation}
\left\{
\begin{array}{ll}
\mathscr K_{m,i}^+=x_i\\
\\
\mathscr K_{m,i}^-=x_i\,\frac{\partial^2}{\partial {x_i}^2} +\frac{m}{2}\frac{\partial}{\partial {x_i}} \\
\\
\mathscr K_{m,i}^0=x_i \, \frac{\partial}{\partial {x_i}} + \frac{m}{4}
\end{array}
\right.
\end{equation}
and the corresponding discrete operators
\begin{equation}
\left\{
\begin{array}{ll}
\mathrm K_{m,i}^+ f(k_i)  = \left(k_i+ \frac{m}{2}-1\right)  f(k_i-1) \\
\\
\mathrm K_{m,i}^-  f(k_i)  = (k_i+1) f(k_i+1)\\
\\
\mathrm K_{m,i}^o f(k_i)  = \left(k_i+\frac{m}{4}\right)  f(k_i)\;.
\end{array}
\right.
\\
\end{equation}
The generator of the BEP(m) then reads
\begin{equation}
\label{abstractbep}
\mathscr{L}_{m} = \frac{1}{2}
\sum_{1\le i < j \le d}
\left(
  \mathscr K_{m,i}^+\mathscr K_{m,j}^- +
  \mathscr K_{m,i}^-\mathscr K_{m,j}^+ -
2\mathscr K_{m,i}^o\mathscr K_{m,j}^o +
  \frac{m^2}{8}
\right)\;,
\end{equation}
By proposition \ref{su11prop}, combined with
theorem \ref{generalprop}, we find that this operator
is dual to
the operator
\begin{equation}
\label{abstractsip}
\mathrm{L}_{m} = \frac{1}{2}
\sum_{1\le i < j \le d}
\left(
  \mathrm K_{m,i}^+\mathrm K_{m,j}^- +
  \mathrm K_{m,i}^-\mathrm K_{m,j}^+ -
2\mathrm K_{m,i}^o\mathrm K_{m,j}^o +
  \frac{m^2}{8}
\right)\;,
\end{equation}
This operator is exactly the generator of the SIP(m).
The duality function is given by, using once more theorem \ref{generalprop}, item 7:
$$
{D}_N(x,k) = \prod_{i=1}^{d} d(x_i,k_i)\;.
$$
where $d(z,k)$ is given in \eqref{dualsu11}.
The multiplicative constant $\Gamma(m/2)$ in \eqref{dualsu11} can be dropped, and the result
of the theorem thus follows
from combining the duality between BEP(m) and SIP(m) with
proposition \ref{bep=wf} and proposition \ref{sip=mor}.
\epr

\subsection{\col{Limiting duality between $d$-types Wright-Fisher diffusion and $d$-types Moran model}}
\label{sub2}

We can now also let $m\to 0$, or correspondingly $\theta\to 0$ to obtain
a duality result between the neutral Wright-Fisher diffusion and
the standard Moran model.

Let $d\ge 2$ be an integer denoting the number of types (or alleles)
in a population.
\begin{definition}
The $d$-types Wright-Fisher model is a diffusion process on the simplex $\sum_{i=1}^dx_i=1$
defined by the generator
\be
\label{dWF}
\mathscr{L}_{d}^{WF} g(x) =
\sum_{i=1}^{d-1} \frac12 x_i(1-x_i) \frac{\partial^2 g(x)}{\partial x_i^2}
- \sum_{1\le i < j \le d-1} x_ix_j \frac{\partial^2 g(x)}{\partial x_i \partial x_j}\;.
\ee
\end{definition}
\begin{definition}
The Brownian Energy process with $m=0$ on the complete graph
with $d$ vertices is a diffusion on $\R_+^d$ given by the generator
\be
\label{bep0}
\mathscr{L}_d^{BEP(0)} f(y)=
\frac{1}{2}\sum_{1\le i< j \le d} y_iy_j \left(\frac{\partial}{\partial y_i} - \frac{\partial}{\partial y_j}\right)^2 f(y)\;.
\ee
\end{definition}
\bp\label{bep0=wf}
The generator of the Brownian Energy process with $m=0$ on the
complete graph with $d$ vertices and  with
initial condition $\sum_{i=1}^{d}x_i=1$
does coincide with the generator of the $d$-types Wright-Fisher diffusion.
\ep
\begin{proof}
\col{Similar to the proof of proposition \ref{bep=wf}}
\end{proof}
%
%
\begin{definition}
In the $d$-types Moran model with population size $N$ a pair of individuals of types $i$ and
$j$ are  sampled uniformly at random, one dies with probability $1/2$ and the other reproduces.
Therefore, denoting type occurrences by $k= (k_1,\ldots,k_{d-1})$, where $k_i$ is the number of individuals of type $i$,
the process has generator
\begin{eqnarray}
\label{dMoran}
\mathscr{L}_{N,d}^{Mor} g(k) & = &
\frac12 \sum_{1\le i < j \le d-1} k_i k_j \;(g(k+e_i-e_j) + g(k-e_i+e_j) - 2g(k))\nonumber \\
& + &
\frac12 \sum_{i=1}^{d-1} k_i\left(N-\sum_{j=1}^{d-1}k_j\right) (g(k+e_i)+g(k-e_i) -2g(k))\;.\nonumber\\
\end{eqnarray}
\end{definition}
\begin{definition}
The Symmetric Inclusion process with $m=0$ on the complete graph with $d$ vertices is a Markov
process on \colk{$\N_0^d$} with  generator
\begin{eqnarray}
\label{sip0}
\mathscr{L}_{d}^{SIP(0)} f(k) & = &
\frac12 \sum_{1\le i < j \le d} k_i k_j (f(k+e_i-e_j) + f(k-e_i+e_j) - 2f(k))\;.\nonumber\\
\end{eqnarray}
\end{definition}
\bp\label{sip0=moran}
The generator of the Symmetric Inclusion process with $m=0$ on the complete graph with
$d$ vertices and with initial condition $\sum_{i=1}^{d}n_i=N$
does coincide with the generator of the $d$-types Moran model
with population size $N$, i.e.
$$
\mathscr{L}_d^{SIP(0)} f(k_1,\ldots,k_{d-1}, k_d)
=
\mathscr{L}_{N,d}^{Mor}g(k_1,\ldots,k_{d-1})
$$
with
$$
g(k_1,\ldots,k_{d-1}) = f(k_1,\ldots,k_{d-1}, N-\sum_{j=1}^{d-1}k_j)\;.
$$
\ep
\begin{proof}
\col{Similarly to the proof of proposition \ref{sip=mor}}, the result follows from the conservation law, namely the
fact that the SIP evolution conserves the total number of particles $k_1+\ldots+ k_d$.
\end{proof}

In the duality result of theorem \ref{pippo3} we cannot directly substitute $m=0$ because
\col{there would be problems when some $k_i=0$.}
To state a duality result for $\col{\theta=0}$, i.e., between the Wright Fisher diffusion and
the Moran model without mutation, we
\col{start again from the duality between Brownian Energy process and Symmetric Inclusion process:}
\beq\label{bim}
&&\E_x^{BEP(m)} \left(\prod_{i=1
}^d
\frac{x_i(t)^
{\xi_i}}{\frac{m}{2} \ldots \left(\frac{m}{2}+\xi_i-1\right)}
\right)
\\
&=&
\E_\xi^{SIP(m)}
\left(\prod_{i=1
}^d
\frac{x_i^
{\xi_i(t)}}{\frac{m}{2} \ldots \left(\frac{m}{2}+\xi_i(t)-1\right)}
\right)\;. \nonumber
\eeq
Here the products in lhs and rhs are by definition equal to $1$ when
$\xi_i=0, \xi_i(t)=0$.\color{black}

For \colk{$\xi\in\N_0^d$}, denote $\caR(\xi)=\sharp \{i \in \{1,\ldots, d\}: \xi_i\geq 1\}$. Then we can rewrite
\eqref{bim}
and obtain
\beq\label{bam}
&&\E_x^{BEP(m)} \left(\prod_{i=1, \xi_i\geq 1}^d
\frac{x_i(t)^
{\xi_i}}{\left(\frac{m}{2}+1\right) \ldots \left(\frac{m}{2}+\xi_i-1\right)}
\right)
\\
&=&
\E_\xi^{SIP(m)}
\left(\color{black}\left(\frac{m}{2}\right)^{\left(\caR(\xi)-\caR(\xi(t))\right)} \color{black} \prod_{i=1, \xi_i(t)\geq 1}^d
\frac{x_i^
{\xi_i(t)}}{\left(\frac{m}{2}+1\right) \ldots \left(\frac{m}{2}+\xi_i(t)-1\right)}
\right)\;, \nonumber
\eeq
\color{black}
where now the denominators in the products in lhs and rhs are by definition equal to $1$ when
$\xi_i=1, \xi_i(t)=1$.\color{black}

Now we are in the position
to take the limit $m\to 0$ and we find
\beq\label{dualsipbepm=0}
&&\E_x^{BEP(0)}
\left(\prod_{i=1, \xi_i\geq 1}^d
\frac{x_i(t)^{{\xi_i}}}{(\xi_i-1)!}\right)
\\
&=&
\lim_{m\to 0}\E_\xi^{SIP(m)}
\left(
\left(\frac{m}{2}\right)^{\left(\caR (\xi)-\caR(\xi(t))\right)}
\prod_{i:\xi_i(t)\geq 1}
\frac{{x_i}^{\xi_i(t)}}{(\xi_i(t)-1)!}
\right)\;.\nonumber
\eeq
Notice that the lhs becomes zero as soon as one of the $x_i$ is zero,
which corresponds to the fact that for all $i$, $x_i=0$ is an absorbing set in the diffusion.
Corresponding to this, the rhs becomes zero as soon as one of the species
disappears, i.e., as soon as $\caR(\xi)$ decreases by one unit. Notice however that in the rhs
we can not simply substitute $m=0$ as we did in the lhs, since
$\left(\frac{m}{2}\right)^{\left(\caR (\xi)-\caR(\xi(t))\right)}$ can be of order $(1/m)^k$
\color{black} with $k>0$ \color{black}
with correspondingly small probability. Therefore, in the rhs we do not exactly
recover the $SIP(0)$, but have to keep $m$ positive and take the limit after the expectation.
We call \eqref{dualsipbepm=0} ``{\em a limiting duality relation with duality function}''
\be\label{dualfunctionm=0}
D(\xi,x)=\left(\prod_{i=1, \xi_i\geq 1}^d
\frac{x_i^{\xi_i}}{(\xi_i-1)!}\right)\;.
\ee
By the correspondence of $SIP(m)$ with the Moran model, and
$BEP(m)$ with the Wright-Fisher diffusion,
the limiting duality relation \eqref{dualsipbepm=0} can also be read as a limiting duality between
Wright-Fisher without mutation and the Moran model in the limit of zero mutation.
\color{black}

\br
We remark that the duality results {in subsections \ref{sub1} and \ref{sub2}}
are of a different nature than the usual dualities
between forward process and coalescent. Indeed, we have here duality between two
``forward processes'' (the Wright-Fisher diffusion and the Moran model), which cannot
be obtained from ``looking backwards in time'', the method by which moment-dualities
with the coalescent are usually obtained. In our framework, the dualities with the
coalescent correspond to a change of representation in the Heisenberg
algebra, whereas the dualities between e.g. Wright-Fisher and Moran model
arise from a change of representation in the $SU(1,1)$ algebra.
\er
\color{black}
\subsection{Limiting self-duality of the $d$-types Moran model}
We can push further the \col{$SU(1,1)$} structure behind the
Moran model and deduce self-duality of the process.
\begin{theorem}\label{pippooo}
The $d$-types Moran model with $N$ individuals
and with generator (\ref{dMoran-mut}) is self-dual with duality function
\be
\bar{D}_N(k,\xi) = \prod_{i=1}^{d} \frac{k_i!}{(k_i-\xi_i)!} \frac{\Gamma\left(\frac{2\theta}{d-1}\right)}{\Gamma\left(\xi_i+\frac{2\theta}{d-1}\right)} \;,
\ee
where $k_d= N-\sum_{i=1}^{d-1} k_i$ and $ \xi_d= N-\sum_{i=1}^{d-1} \xi_i$.
\end{theorem}
\begin{proof}
The result follows from the self-duality property of the SIP(m) process \cite{grv}
and from proposition \ref{sip=mor}.
\end{proof}

The limit $m\to 0$, or equivalently $\theta\to 0$,
leads to a limiting self-duality relation, i.e., the $SIP(0)$ is dual
to $SIP(m)$ in the limit $m\to 0$, and correspondingly, the Moran model
with zero mutation has a limiting self-duality relation with the Moran model
in the limit of zero mutation
\beq\label{boembalaa}
&&\E^{SIP(0)}_\eta
\left(
\prod_{i=1, \xi_i\geq 1}^d
\frac{\eta_i(t) !}{(\eta_i(t)-\xi_i)!(\xi_i-1)!}
\right)
\\
&=&
\lim_{m\to 0}\E^{SIP(m)}_\xi
\left(\left(\frac{m}{2}\right)^{\caR(\xi)-\caR(\xi(t))}
\prod_{i=1, \xi_i(t)\geq 1}^d
\frac{\eta_i !}{(\eta_i-\xi_i(t))!(\xi_i(t)-1)!}
\right)\;. \nonumber
\eeq
\color{black}
\color{black}

\subsection{Examples}
\colora{Here we give a few simple illustrations of concrete computations using the dualities of the present section. 
Notice that  these computations can also be performed  using
the coalescent in the context of $d$ types.
One can for example use the function-valued dual process of the Fleming-Viot process, discussed e.g. in Sect. 1.12 of \cite{ethe-super}
or in Sect. 2.8 of \cite{dawson-dd}, or alternatively the embedding of the Kingman coalescent into the Òlook-downÓ construction 
\cite{donnelly-kurtz,bertoin-legall}.} 

Our approach using
the duality between Wright-Fisher and Moran is an alternative, and in our opinion slightly simpler
one.
\color{black}
In general, remark that the duality between Wright-Fisher and Moran implies that if we want to
compute an expectation of a polynomial of degree $k$ at time
$t$ in the multi-type Wright Fisher model, we have to  consider a Moran model with
$k$ individuals. Also, if we want to compute the expectation of a polynomial of degree
$k$ in the number of individuals of different types in a Moran model with $N$ individuals,
we can do it by using only a Moran model with $k$ individuals. So the main
simplification coming from these dualities is the fact that we can go from ``many'' ($N$)
to ``few'' ($k$) individuals (which can be useful in particular in simulations).
\color{black} The concrete computations that follow below are chosen somewhat arbitrarily
as an illustration \colora{of those simplifications.}

\color{black}
Before we start these computations, we remark that if
in \eqref{boembalaa} or \eqref{dualsipbepm=0}, we start with $\caR(\xi)=d$ equal to its maximal
value, then
the non-zero contributions in the limit $m\to 0$ only come from
configuration $\caR(\xi_t)= d$ (since automatically $\caR(\xi_t)\leq d$, so in that case there
are no contributions for which $\caR(\xi_t)>\caR(\xi)$, i.e., with a negative exponent of $m$).
\color{black}
\ben
\item Heterozygosity of two-types Wright-Fisher \col{diffusion}.
This is defined as the probability that two randomly chosen individuals
are of different types (\cite{etheridge}, pag. 48).
To compute this quantity we can use 
the limiting duality between
the $BEP(0)$ process $(x(t),y(t))$ on two sites, with initial condition
$(x,y)$ such that $x+y=1$ and the process $SIP(m)$, $m\to 0$. \color{black}
\begin{eqnarray*}
\E^{BEP(0)}_{x,y} (x(t)y(t))&=&  \lim_{m\to 0}\E^{SIP(m)}_{1,1}\left( xy I(n_1(t)=1, n_2(t)=1)\right)\\
&=&
\E^{SIP(0)}_{1,1}\left( xy I(n_1(t)=1, n_2(t)=1)\right)
\\
&=& xy
\pee_{1,1}(n_1(t)=1, n_2(t)=1)= xy e^{-t}\;,
\end{eqnarray*}
\col{where $\pee_{1,1}$ denotes the law of the $SIP(0)$ process initialized
with one particle per site.}
\item Higher moments of two-types Wright-Fisher \col{diffusion}. We use the same notation of
the previous item and \col{consider for instance $x^2y$}. Further, we notice that if
we start the $SIP(0)$ from initial configuration $(n_1,n_2)=(2,1)$, then
\colora{configurations $(3,0)$ and $(0,3)$ are absorbing and}
the only transitions before absorption are of the type $(2,1)\to (1,2)$ and vice
versa, and both transitions occur at rate $2$, whereas from any
of these states, the rate to go to the absorbing states is also equal to 2.
Therefore,
\begin{eqnarray*}
&&\E^{BEP(0)}_{xy} (x^2(t) y(t))
\\
&&=x^2 y \pee_{2,1}^{SIP(0)} ((n_1(t), n_2(t))=(2,1)) + \nonumber \\
&& \qquad + x y^2 \pee_{2,1}^{SIP(0)} ((n_1(t), n_2(t))=(1,2))
\\
&&=
\frac{e^{-2t}}{2}(x^2y (1+e^{-2t}) + xy^2 (1-e^{-2t}))\;.
\end{eqnarray*}
\item Analogue of heterozygosity for $d$-types Wright Fisher diffusion. Notice that
for multitype Wright Fisher, there is no simple analogue of the Kingman's coalescent,
as for the two-types case.
This means that we have the $BEP(0)$ started from $x_1,\ldots,x_d$
\begin{eqnarray*}
&&\E^{BEP(0)}_{x_1,\ldots,x_d} (x_1(t)\ldots x_d(t))
\\
&&=x_1\ldots x_d\E^{SIP(0)}_{(1,1,\ldots,1)} \left(I(n_i(t)\not= 0 \;\;\forall i\in \{1,\ldots,d\})\right)
\\
&&= x_1\ldots x_d e^{-\frac{d(d-1)t}{2}}\;.
\end{eqnarray*}
\item Analogue of $x^2y$ for the multi-type case.
\begin{eqnarray*}
&&\E^{BEP(0)}_{x_1,\ldots,x_d} (x^2_1(t) x_2\ldots x_d(t))
\\
&&=\E^{SIP(0)}_{(2,1,\ldots,1)} \left(\left(\prod_i x_i^{n_i(t)}\right)I(n_i(t)\not= 0 \;\;\forall i\in \{1,\ldots,d\})\right)
\\
&&= \sum_{i=1}^d \left(\prod_{j\not=i} x_j\right) x_i^2
\pee^{SIP(0)}_{(2,1,\ldots,1)} (n_1(t)=1,\ldots, n_i(t)=2, \ldots, n_d(t)=1)\;.
\end{eqnarray*}
To compute the latter probability, we remark that starting from
the configuration $(2,1,\ldots,1)$, the $SIP(0)$ will \colora{reach an absorbing state}
as soon as one of
the particles on the sites with a single occupation makes a jump,
which happens at rate $(d-1)(d-2)+ (d-1)2= d(d-1)$. Further, as
long as absorption did not occur, the site with two particles
moves as a continuous-time random walk $X^d_t$ on the complete graph
of $d$ vertices, moving at rate $2$ and starting at site $1$.
Therefore
\begin{eqnarray*}
&&\pee^{SIP(0)}_{(2,1,\ldots,1)} (n_1(t)=1,\ldots, n_i(t)=2, \ldots, n_d(t)=1)
\\
&&= e^{-d(d-1) t} \pee (X^d_t=i)
\\
&&=
e^{-2dt} +\frac1d (1- e^{-2dt}) \delta_{i,1}+ (1-\delta_{i,1})\frac1d (1- e^{-2dt})\;.
\end{eqnarray*}
\een
As these examples illustrate, computations of (appropriately chosen)
moments in the multi-type Wright Fisher diffusion
reduce to finite dimensional Markov chain computations, associated to inclusion
walkers on the complete graph until absorption, which occurs as soon as a site becomes empty.
The same can be done for the multi-type Moran model, using its self-duality.

\vskip.5cm
{
{\bf Acknowledgments.}
We acknowledge financial support from  the Italian Research Funding
Agency (MIUR) through FIRB project ``Stochastic processes in interacting particle
systems: duality, metastability and their applications'', grant n. RBFR10N90W and
the Fondazione Cassa di Risparmio Modena through the International Research 2010 project.
}

\end{document}